\documentclass[12pt]{article}
\usepackage[utf8]{inputenc}
\usepackage{amsmath,amssymb,amsthm}
\usepackage[inline]{enumitem}
\usepackage[nocompress, noadjust]{cite}
\usepackage{color}
\usepackage{tikz}
\usetikzlibrary{patterns}
\usepackage{subcaption}
\usepackage{authblk}
\usepackage{hyperref}

\title{Zero-sum squares in $\{-1, 1\}$-matrices with low discrepancy}
\author{Tom Johnston}
\affil{\small Mathematical Institute, University of Oxford, United Kingdom.\\ \href{mailto:thomas.johnston@maths.ox.ac.uk}{thomas.johnston@maths.ox.ac.uk} }
\date{}

\newtheorem{theorem}{Theorem}

\newtheorem{lemma}[theorem]{Lemma}

\newtheorem{problem}{Problem}
\newtheorem{claim}{Claim}

\newtheorem{observation}[theorem]{Observation}
\newtheorem{conjecture}[theorem]{Conjecture}

\newcommand{\floor}[1]{\left\lfloor #1 \right\rfloor}
\newcommand{\ceil}[1]{\left\lceil #1 \right\rceil}

\DeclareMathOperator{\disc}{disc}

\definecolor{blue3}{HTML}{64B5F6}
\definecolor{blue5}{HTML}{2196F3}
\definecolor{blue7}{HTML}{1976D2}

\definecolor{lightblue7}{HTML}{0288D1}
\definecolor{lightblue5}{HTML}{03A9F4}
\definecolor{lightblue3}{HTML}{4FC3F7}

\definecolor{yellow5}{HTML}{FFEB3B}
\definecolor{yellow7}{HTML}{FBC02D}

\begin{document}
\maketitle

\begin{abstract}
	Given a matrix $M = (a_{i,j})$ a square is a $2 \times 2$ submatrix with entries $a_{i,j}$, $a_{i, j+s}$, $a_{i+s, j}$, $a_{i+s, j +s}$ for some $s \geq 1$, and a zero-sum square is a square where the entries sum to $0$. Recently, Ar\'evalo, Montejano and Rold\'an-Pensado proved that all large $n \times n$ $\{-1,1\}$-matrices $M$ with  discrepancy $|\sum a_{i,j}| \leq n$ contain a zero-sum square unless they are split. We improve this bound by showing that all large $n \times n$ $\{-1,1\}$-matrices $M$ with discrepancy at most $n^2/4$ are either split or contain a zero-sum square. Since zero-sum square free matrices with discrepancy at most $n^2/2$ are already known, this bound is asymptotically optimal.
\end{abstract}

\section{Introduction}

A \emph{square} $S$ in a matrix $M= \left( a_{i,j} \right)$ is a $2 \times 2$ submatrix of the form \[S = \left(\begin{array}{cc}
			a_{i,j}   & a_{i, j+s}   \\
			a_{i+s,j} & a_{i+s, j+s}
		\end{array}\right).
\]

In 1996 Erickson \cite{erickson1996introduction} asked for the largest $n$ such that there exists an $n \times n$ binary matrix $M$ with no squares which have constant entries. An upper bound was first given by Axenovich and Manske \cite{axenovich2008monochromatic} before the answer, 14, was determined by Bacher and Eliahou in \cite{bacher2010extremal}.

Recently, Ar\'evalo, Montejano and Rold\'an-Pensado \cite{arevalo2021zero} initiated the study of a zero-sum variant of Erickson's problem. Here we wish to avoid \emph{zero-sum squares}, squares with entries that sum to $0$.

Zero-sum problems have been well-studied since the Erd\H{o}s-Ginsburg-Ziv Theorem in 1961 \cite{erdos1961theorem}, which says that any set of $2n-1$ integers must contain a set of $n$ integers which sum to $0$ modulo $n$. Much of the research has been on zero-sum problems in finite abelian groups (see the survey \cite{gao2006zero} for details), but problems have also been studied in other settings such as on graphs (see e.g. \cite{caro2016ero, caro2019zero, caro2020zero, bialostocki1993zero}). Of particular relevance is the work of Balister, Caro, Rousseau and Yuster in \cite{balister2002zero} on submatrices of integer valued matrices where the rows and columns sum to $0 \mod p$, and the work of Caro, Hansberg and Montejano on zero-sum subsequences in bounded sum $\{-1,1\}$-sequences \cite{caro2019zerosum}.

Given an $n \times m$ matrix $M = \left( a_{i,j} \right)$ define the \emph{discrepancy} of $M$ as the sum of the entries, that is,
\[\disc(M) = \sum_{\substack{1 \leq i \leq n\\1 \leq j \leq m}} a_{i,j}. \]
We say a square $S$ is a \emph{zero-sum square} if $\disc(S) = 0$, or equivalently,
\[a_{i,j} + a_{i, j+s} + a_{i+s,j} + a_{i+s, j+s} = 0.\]
We will be interested in $\{-1,1\}$-matrices $M$ which do not contain any zero-sum squares. Clearly, matrices with at most one $-1$ cannot contain a zero-sum square and, in general, there are many such matrices when the number of $-1$s is low. But what happens if there are a similar number of 1s and $-1$s? In particular, what happens if the matrix $M$ is itself zero-sum?

An $n \times m$ $\{-1,1\}$-matrix $M = \left(a_{i,j}\right)$ is said to be \emph{$t$-split} for some $0 \leq t \leq n +m -1$ if
\[a_{i,j} = \begin{cases}
		1  & i + j \leq t + 1, \\
		-1 & i + j \geq t + 2.
	\end{cases}\]
Note that when $t = 0$ the matrix consists entirely of $-1$ entries, and when $t = n + m - 1$ the matrix consists entirely of $+1$ entries.
We say a matrix $M$ is \emph{split} if there is some $t$ such that a $t$-split matrix $N$ can be obtained from $M$ by applying vertical and horizontal reflections.
Split matrices are of particular interest since they can have low absolute discrepancy, yet they never contain a zero-sum square. However, it is not hard to check that an $n \times n$ split matrix cannot have discrepancy 0, and it may still be the case that a zero-sum matrix $M$ must contain a zero-sum square.

This was confirmed by Ar\'evalo, Montejano and Rold\'an-Pensado in \cite{arevalo2021zero}. In fact, they proved that, except when $n \leq 4$, every $n \times n$ non-split $\{-1,1\}$-matrix $M$ with $|\disc(M)| \leq n$ has a zero-sum square. They remark that it should be possible to extend their proof to give a bound of $2n$, and they conjecture that the bound $Cn$ should hold for any $C > 0$ when $n$ is large enough relative to $C$.

\begin{conjecture}[Conjecture 5 in \cite{arevalo2021zero}]
	For every $C > 0$ there is an integer $N$ such that whenever $n \geq N$ the
	following holds: every $n \times n$ non-split $\{-1, 1\}$-matrix $M$ with $|\disc(M)| \leq Cn$
	contains a zero-sum square.
\end{conjecture}

Let $f(n)$ be the absolute value of the minimum discrepancy of a non-split $\{-1,1\}$-matrix with no zero-sum squares. Ar\'evalo, Montejano and Rold\'an-Pensado proved that $f(n) \geq n + 1$ for all $n \geq 5$, and the conjecture would imply that $f(n) = \omega(n)$. We improve the lower bound on $f$ to $\floor{n^2/4} + 1$ (for all $n \geq 5$), showing that $f = \Omega(n^2)$.

\begin{theorem}\label{thm:low-bound}
	Let $n \geq 5$. Every $n \times n$ non-split $\{-1,1\}$-matrix $M$ with $|\disc(M)| \leq n^2/4$ contains a zero-sum square.
\end{theorem}

The best known construction for a non-split matrix with no zero-sum squares has discrepancy close to $n^2/2$, about twice the lower bound given here, and our computer experiments suggest that this construction is in fact optimal. Although the lower bound now only differs from the upper bound by a constant factor, closing the gap between the upper and lower bounds remains a very interesting problem and we discuss it further in Section \ref{sec:open-problems}.

\section{Proof}
For  $p\leq r$ and $q \leq s$ define the \emph{consecutive submatrix} $M[p:r, q:s]$ by
\[M[p:r, q:s] = \left(\begin{array}{cccc}
			a_{p,q}    & a_{p, q+1}   & \dotsb & a_{p, s}   \\
			a_{p+1, q} & a_{p+1, q+1} & \dotsb & a_{p+1, s} \\
			\vdots     & \vdots       & \ddots & \vdots     \\
			a_{r, q}   & a_{r+1, q}   & \dotsb & a_{r,s}
		\end{array}
	\right).
\]
Throughout the rest of this paper, we will assume that all submatrices except squares are consecutive submatrices.

To show that every zero-sum $n \times n$ $\{-1,1\}$-matrix (where $n \geq 5$) contains a zero-sum square (for $n \geq 5$), Ar\'evalo, Montejano and Rold\'an-Pensado prove that a small $t'$-split submatrix $M'$ determines many entries of the matrix $M$, and their proof leads to the following lemma. An example application is shown in Figure \ref{fig:struct}.

\begin{lemma}[\cite{arevalo2021zero}]
	\label{lem:struct}
	Let $M$ be an $n \times n $ $\{-1,1\}$-matrix with no zero-sum squares, and suppose that there is an $s \times s$ submatrix $M' = M[p: p+s-1, q: q+s-1]$ which is $t'$-split for some $2 \leq t' \leq 2s - 3$. Let $t = t' + p + q -2$, $T = \floor{3t/2}$ and suppose $t \leq n$.
	\begin{enumerate}
		\item The submatrix \[N = M\left[1: \min\{T, n\}, 1:\min\{T, n\}\right]\] is $t$-split.
	\end{enumerate}

	Furthermore, both $a_{i,j} = 1$ and $a_{j,i} = 1$ whenever $T < j \leq T + t -2$ and one of the following holds:
	\begin{enumerate}
		\setcounter{enumi}{1}
		\item $j - t < i \leq T$,
		\item $i \leq \floor{\frac{T+t+1-j}{2}}$,
		\item $i = j$.
	\end{enumerate}
\end{lemma}

		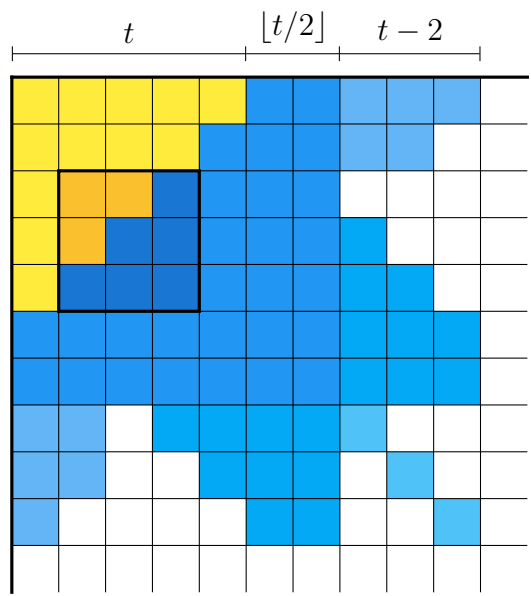
\begin{figure}
			\centering
			\begin{tikzpicture}[scale=0.5\textwidth/11cm]

				\fill[yellow5] (0,10) rectangle +(5,1);
				\fill[yellow5] (0,9) rectangle +(4,1);
				\fill[yellow5] (0,8) rectangle +(1,1);
				\fill[yellow5] (0,7) rectangle +(1,1);
				\fill[yellow5] (0,6) rectangle +(1,1);
				
				\fill[yellow7] (1,8) rectangle +(2,1);
				\fill[yellow7] (1,7) rectangle +(1,1);
				
				\fill[blue7] (3, 8) rectangle +(1,1);
				\fill[blue7] (2, 7) rectangle +(2,1);
				\fill[blue7] (1, 6) rectangle +(3,1);
				
				\fill[blue5] (4, 6) rectangle +(1, 4);
				\fill[blue5] (0, 4) rectangle +(7, 2);
				\fill[blue5] (5, 6) rectangle +(2, 5);
				
				\fill[blue3] (7, 10) rectangle +(3, 1);
				\fill[blue3] (7, 9) rectangle +(2,1);
				\fill[blue3] (0, 1) rectangle +(1, 3);
				\fill[blue3] (1,2) rectangle +(1, 2);
				
				\fill[lightblue5] (3, 3) rectangle +(4, 1);
				\fill[lightblue5] (4, 2) rectangle +(3, 1);
				\fill[lightblue5] (5, 1) rectangle +(2, 1);
				
				\fill[lightblue5] (7,4) rectangle +(1, 4);
				\fill[lightblue5] (8,4) rectangle +(1, 3);
				\fill[lightblue5] (9,4) rectangle +(1, 2);
				
				\fill[lightblue3] (7, 3) rectangle +(1,1);
				\fill[lightblue3] (8, 2) rectangle +(1,1);
				\fill[lightblue3] (9, 1) rectangle +(1,1);
				
				\draw[black] (0,1) -- +(11,0);
				\draw[black] (0,2) -- +(11,0);
				\draw[black] (0,3) -- +(11,0);
				\draw[black] (0,4) -- +(11,0);
				\draw[black] (0,5) -- +(11,0);
				\draw[black] (0,6) -- +(11,0);
				\draw[black] (0,7) -- +(11,0);
				\draw[black] (0,8) -- +(11,0);
				\draw[black] (0,9) -- +(11,0);
				\draw[black] (0,10) -- +(11,0);
				\draw[black, very thick, cap=rect] (0,11) -- +(11,0);
				
				\draw[very thick, black, cap=rect] (0,0) -- +(0,11);
				\draw[black] (1,0) -- +(0,11);
				\draw[black] (2,0) -- +(0,11);
				\draw[black] (3,0) -- +(0,11);
				\draw[black] (4,0) -- +(0,11);
				\draw[black] (5,0) -- +(0,11);
				\draw[black] (6,0) -- +(0,11);
				\draw[black] (7,0) -- +(0,11);
				\draw[black] (8,0) -- +(0,11);
				\draw[black] (9,0) -- +(0,11);
				\draw[black] (10,0) -- +(0,11);
				
				\draw[very thick] (1,6) rectangle +(3,3);
				
				\draw[|-|] (0, 11.5) -- +(5, 0);
				\draw[-|] (5,11.5) -- +(2,0);
				\draw[-|] (7, 11.5) -- +(3,0);
				
				\draw node[above] at (2.5, 11.5) {$t$};
				\draw node[above] at (6, 11.5) {$\floor{t/2}$};
				\draw node[above] at (8.5, 11.5) {$t-2$};

			\end{tikzpicture}
		
		\caption[The entries known from applying Lemma \ref{lem:struct}.]{The entries known from applying Lemma \ref{lem:struct}. The yellow squares represent $-1$s and the blue squares represent $1$s. The submatrix $M'$ is shown in a darker shade.}
		\label{fig:struct}
		\end{figure}
Note that we can apply this lemma even when it is a reflection of $M'$ which is $t$-split; we just need to suitably reflect $M$ and potentially multiply by $-1$, and then undo these operations at the end. The matrix $N$ will always contain at least one of $a_{1,1}$, $a_{1,n}$, $a_{n,1}$ and $a_{n,n}$, and if $N$ contains two, then $M$ is split.

We will also make use of the following observation. This will be used in conjunction with the above lemma to guarantee the existence of some additional $1$s, which allows us to show a particular submatrix has positive discrepancy.

\begin{observation}
	\label{obs:oneof}
	Let $M$ be an $n \times n $ $\{-1,1\}$-matrix with no zero-sum squares, and suppose that $a_{i,i} = 1$ for every $i \in [n]$. Then at least one of $a_{i,j}$ and $a_{j,i}$ is 1. In particular, $a_{i,j} + a_{j,i} \geq 0$ for all $1 \leq i ,j \leq n$.
\end{observation}

The final lemma we will use to prove Theorem \ref{thm:low-bound} is a variation on Claim 11 from \cite{arevalo2021zero}. The main difference between Lemma \ref{lem:submatrix} and the result used by  Ar\'evalo, Montejano and Rold\'an-Pensado is that we will always find a square submatrix, which simplifies the proof of Theorem \ref{thm:low-bound}.

\begin{lemma}
	\label{lem:submatrix}
	For $n \geq 8$, every $n \times n $ $\{-1,1\}$-matrix $M$ with $|\disc(M)| \leq n^2/4$ has an $n' \times n'$ consecutive submatrix $M'$ with $|\disc(M')| \leq (n')^2/4$ for some $(n-1)/2 \leq n' \leq (n+1)/2$.
\end{lemma}
\begin{proof}
	We only prove this in the case $n$ is odd as the case $n$ is even is similar, although simpler.
	Partition the matrix $M$ into 9 regions as follows. Let the four $(n-1)/2 \times (n-1)/2$ submatrices containing $a_{1,1}$, $a_{1,n}$, $a_{n,n}$ and $a_{n,1}$ be $A_1, \dots, A_4$ respectively. Let the $(n-1)/2 \times 1$ submatrix between $A_1$ and $A_2$ be $B_1$ and define $B_2$, $B_3$ and $B_4$ similarly. Finally, let the central entry be $B_5$. The partition is shown in Figure \ref{fig:regions-part}.

			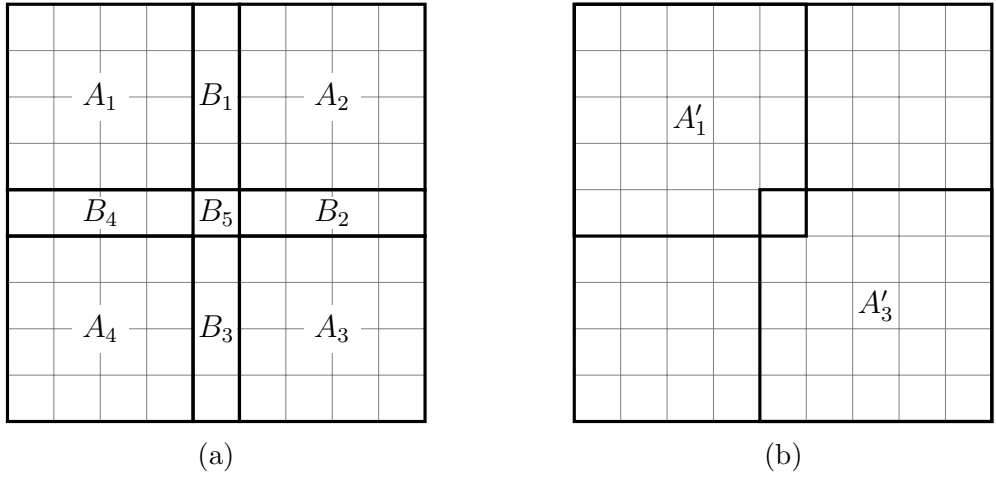
\begin{figure}
		\centering
		\begin{subfigure}{0.45\textwidth}
			\centering
			\begin{tikzpicture}[scale=\textwidth/10cm]
				\draw[step=1, very thin, gray] (0,0) grid (9,9);
				\draw[fill=none, stroke=black, very thick] (0,5) rectangle(4,9);
				\draw[fill=none, stroke=black, very thick] (5,5) rectangle (9,9);
				\draw[fill=none, stroke=black, very thick] (0,0) rectangle (4,4);
				\draw[fill=none, stroke=black, very thick] (5,0) rectangle (9,4);
				\draw node[fill=white] at (2, 7) {$A_1$};
				\draw node[fill=white] at (7, 7) {$A_2$};
				\draw node[fill=white] at (7, 2) {$A_3$};
				\draw node[fill=white] at (2, 2) {$A_4$};
				\draw[fill=none, stroke=black, very thick] (4,5) rectangle(5,9);
				\draw[fill=none, stroke=black, very thick] (5,4) rectangle(9,5);
				\draw[fill=none, stroke=black, very thick] (4,0) rectangle(5,4);
				\draw[fill=none, stroke=black, very thick] (0,4) rectangle(4,5);
				\draw node[fill=white, inner sep=1] at (4.5, 7) {$B_1$};
				\draw node[fill=white, inner sep=1] at (7, 4.5) {$B_2$};
				\draw node[fill=white, inner sep=1] at (4.5, 2) {$B_3$};
				\draw node[fill=white, inner sep=1] at (2, 4.5) {$B_4$};
				\draw node[fill=white, inner sep=1] at (4.5, 4.5) {$B_5$};
				
			\end{tikzpicture}
			\caption{}
			\label{fig:regions-part}
		\end{subfigure}\hfill%
		\begin{subfigure}{0.45\textwidth}
			\centering
			\begin{tikzpicture}[scale=\textwidth/10cm]
				\draw[step=1, very thin, gray] (0,0) grid (9,9);
				
				\draw[fill=none, black, very thick] (4,0) rectangle (9,5);
				\draw[fill=none, black, very thick] (0,4) rectangle(5,9);
				\draw[fill=none, black, very thick] (0,0) rectangle(9,9);
				\draw node[] at (2.5, 6.5) {$A'_1$};
				\draw node[] at (6.5, 2.5) {$A'_3$};

			\end{tikzpicture}
			\caption{}
			\label{fig:regions-overlap}
		\end{subfigure}%
		\caption{A subset of the regions used in the proof of Lemma \ref{lem:submatrix}.}
		\label{fig:regions}
	\end{figure}

	As these partition the matrix $M$, we have
	\begin{equation}\label{eqn:part}
		\disc(M) = \disc(A_1) + \dotsb + \disc(A_4) + \disc(B_1) + \dotsb + \disc(B_5).
	\end{equation}

	Let the overlapping $(n+1)/2 \times (n+1)/2$ submatrices containing $a_{1,1}$, $a_{1,n}$, $a_{n,n}$ and $a_{n,1}$ be $A_1', \dots, A_4'$ (as indicated in Figure \ref{fig:regions-overlap}). The submatrices $B_1, \dots, B_4$ each appear twice in the $A_i'$ and $B_5$ appears four times and, by subtracting these overlapping regions, we obtain a second equation for $\disc(M)$:
	\begin{multline}\label{eqn:part2}
		\disc(M) = \disc(A_1') + \dotsb + \disc(A_4')\\ - \disc(B_1) - \dotsb - \disc(B_4) - 3 \disc(B_5).
	\end{multline}

	If any of the $A_i$ or $A_i'$ have $|\disc(A_i)| \leq (n-1)^2/16$ or $|\disc(A_i')| \leq (n+1)^2/16$ respectively, we are done, so we may assume that this is not the case.
	First, suppose that $\disc(A_i) > (n-1)^2/16$ and $\disc(A_i') > (n+1)^2/16$ for all $i = 1,2,3,4$. Since $n - 1$ is even and $\disc(A_i) \in \mathbb{Z}$, we must have $\disc(A_i) \geq (n-1)^2/16 + 1/4$, and similarly, $\disc(A_i') \geq (n+1)^2/16 + 1/4$. Adding the equations (\ref{eqn:part}) and (\ref{eqn:part2}) we get the bound
	\[n^2/2 \geq 2 \disc(M) \geq (n+1)^2/4 + (n-1)^2/4 + 2 - 2 \disc(B_5), \]
	which reduces to $\disc(B_5) \geq 5/4$. This gives a contradiction since $B_5$ is a single square. Similarly we get a contradiction if, for every $i$, both $\disc(A_i) < - (n-1)^2/16$ and $\disc(A_i') < - (n+1)^2/16$.

	This only leaves the case where two of the 8 submatrices have different signs. If $A_i' > (n+1)^2/16$, then, for $n \geq 8$, \[A_i > (n+1)^2/16 - n > -(n-1)^2/16,\] and either $|\disc(A_i)| \leq (n-1)^2/16$, a contradiction, or  $\disc(A_i) > 0$. By repeating the argument when $\disc(A_i')$ is negative, it follows that $A_i$ and $A_i'$ have the same sign for every $i$. In particular, two of the $A_i'$ must have different signs, and we can apply an interpolation argument as in \cite{arevalo2021zero}.

	Without loss of generality, we can assume that $\disc(A_1') > (n+1)^2/16$ and $\disc(A_2') < -(n+1)^2/16$. Consider the sequence of matrices $N_0, \dots, N_{(n-1)/2} $ where \[N_i = M[1: (n+1)/2, i + 1: i + (n+1)/2].\]
	We claim that there is a $j$ such that $|\disc(N_j)| \leq (n+1)^2/16$, which would complete the proof of the lemma. By definition, $N_0 = A_1'$ and $N_{(n-1)/2} = A_2'$ so there must be some $j$ such that $\disc(N_{j-1}) > 0$ and $\disc(N_j) \leq 0$. Since the submatrices $N_{j-1}$ and $N_{j}$ share most of their entries $|\disc(N_{j-1}) - \disc(N_j)| \leq n+1$ and, as $(n+1)^2/8 > (n+1)$, it cannot be the case that $\disc(N_{j-1}) > (n+1)^2/16$ and $\disc(N_j) < -(n+1)^2/16$. This means there must be some $j$ such that $|\disc(N_j)| \leq (n+1)^2/16$, as required.
\end{proof}

Armed with the above results, we are now ready to prove our main result, but let us first give a sketch of the proof which avoids the calculations in the main proof.

\begin{proof}[Sketch proof of Theorem \ref{thm:low-bound}]
	Assume we have an $n \times n$ $\{-1,1\}$-matrix $M$ with no zero-sum squares and which has $|\disc(M)| \leq n^2/4$. We will prove the result by induction, so we assume that the result is true for $5 \leq n' < n$.

	Applying Lemma \ref{lem:submatrix} gives a submatrix $M'$ with low discrepancy. Since $M'$ also contains no zero-sum squares, we know that it is split by the induction hypothesis. Applying Lemma \ref{lem:struct} then gives a lot of entries $M$ and, in particular, a submatrix $N$ with high discrepancy. Since we are assuming that $M$ has low discrepancy, the remainder $M \setminus N$ of $M$ not in $N$ must either have low discrepancy or negative discrepancy. In both cases we will find $B$, a submatrix of $M$ with low discrepancy. When the discrepancy of $M \setminus N$ is low, we use an argument similar to the proof of Lemma \ref{lem:submatrix}, and when the discrepancy of $M \setminus N$ is negative, we find a positive submatrix using Observation \ref{obs:oneof} and then use an interpolation argument.

	By the induction hypothesis, $B$ must also be split and we can apply Lemma \ref{lem:struct} to find many entries of $M$. By looking at specific $a_{i,j}$, we will show that the two applications of Lemma \ref{lem:struct} contradict each other.
\end{proof}

We now give the full proof of Theorem \ref{thm:low-bound}, complete with all the calculations. To start the induction, we must check the cases $n < 30$ which is done using a computer. The problem is encoded as a SAT problem using PySAT \cite{imms-sat18} and checked for satisfiability with the CaDiCaL solver. The code to do this is attached to the arXiv submission.

\begin{proof}[Proof of Theorem \ref{thm:low-bound}]
	We will use induction on $n$. A computer search gives the result for all $n < 30$, so we can assume that $n \geq 30$ and that the result holds for all $5 \leq n' < n$.

	Suppose, towards a contradiction, that $M$ is an $n \times n$ matrix with no zero-sum squares and $|\disc(M)| \leq n^2/4$ .
	By Lemma \ref{lem:submatrix}, we can find an $n' \times n'$ submatrix $M' = M[p:p+n' - 1, q:q+n' - 1]$ with $(n-1)/2 \leq n ' \leq (n+1)/2$ and $|\disc(M')| \leq (n')^2/4$. By the induction hypothesis and our assumption that $M$ doesn't contain a zero-sum square, the matrix $M'$ must be split. By reflecting $M$ and switching $-1$ and $1$ as necessary, we can assume that the submatrix $M'$ is $t'$-split for some $t'$, and that $t := t' + p + q -2 \leq n$.

	We will want to apply Lemma \ref{lem:struct}, for which we need to check $2 \leq t' \leq 2n' - 3$. If $t' \leq 1$ or $t' \geq 2n' - 2$, then the discrepancy of $M'$ is \[|\disc(M')| \geq  (n')^2 -1 > (n')^2/4,\] which contradicts our choice of $M'$. In fact, since $\disc(M') \leq (n')^2/4$ and $\disc(M') \leq (n')^2 - t'(t'+1)$ we find
	\begin{equation}
		\label{eqn:tbound}
		t \geq t' \geq \frac{1}{2} \left( \sqrt{3(n')^2 + 1} -1 \right) \approx 0.433 n.
	\end{equation}

	If $t + \floor{t/2} \geq n$, the matrix $M$ is $t$-split and we are done, so we can assume that this is not the case, and that $t \leq 2n/3$. We will also need the following bound on $2t + \floor{t/2} -2$, which follows almost immediately from (\ref{eqn:tbound}).

	\begin{claim}\label{claim:tgeqnmins1}
		We have
		\[2t + \floor{t/2} -2 \geq n - 1.\]
	\end{claim}
	\begin{proof}
		Substituting $n' \geq (n-1)/2$ into (\ref{eqn:tbound}) gives the following bound on $t$:
		\[t \geq \frac{1}{4} \left( \sqrt{3n^2 - 6n + 7} - 2 \right).\]
		We now lower bound $\floor{t/2}$ by $(t-1)/2$ to find
		\begin{align*}
			2t + \floor{t/2} -2 & \geq 2t + \frac{t-5}{2}                                \\
			                    & \geq  \frac{5}{8} \sqrt{3n^2 - 6n + 7} - \frac{15}{4}.
		\end{align*}
		The right hand side  grows like $\frac{\sqrt{75}}{8} n$ asymptotically, which is faster than $n$, so the claim is certainly true for large enough $n$. In fact, the equation $ \frac{5}{8} \sqrt{3n^2 - 6n + 7} - \frac{15}{4} \geq n -1$ can be solved explicitly to obtain the following the bound on $n$:
		\[n \geq \frac{1}{11} \left( 251 + 20 \sqrt{166} \right) \approx 46.2.\]
		This still leaves the values $30 \leq n \leq 46$ for which the bounds above are not sufficient. These cases can be checked using a computer.
	\end{proof}

	Let $k = \ceil{5n/6}$ and let $N  = M[1:k, 1:k]$ be the $k \times k$ sub-matrix in the top left corner which contains $a_{1,1}$. We will apply Lemma \ref{lem:struct} and Observation \ref{obs:oneof} to guarantee lots of 1s in $N$, and therefore ensure $N$ has large discrepancy. This will mean that the rest of $M$ which is not in $N$ must have low discrepancy, and we can find another split submatrix $B$.

	\begin{claim}\label{claim:B}
		There is an $(n-k) \times (n-k)$ submatrix $B$ which is disjoint from $N$ and with $|\disc(B)| \leq (n-k)^2/4$.
	\end{claim}
	\begin{proof}
		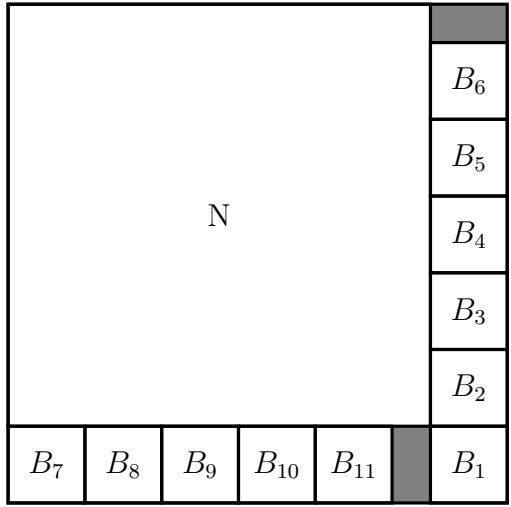
\begin{figure}
			\centering
			\begin{tikzpicture}[scale=0.45\textwidth/(12cm+2.4pt)]

				\fill[gray] (10, 0) rectangle (11, 2);
				\fill[gray] (11, 13) rectangle (13, 12);

				\draw[very thick] (0,0) rectangle (13, 13);
				\draw[very thick] (0, 13) rectangle (11, 2);
				\draw node at (5.5, 7.5) {N};

				\draw[very thick] (0, 0) rectangle (2, 2);
				\draw node at (1,1) {$B_7$};
				\draw[very thick] (2, 0) rectangle (4, 2);
				\draw node at (3,1) {$B_8$};
				\draw[very thick] (4, 0) rectangle (6, 2);
				\draw node at (5,1) {$B_9$};
				\draw[very thick] (6, 0) rectangle (8, 2);
				\draw node at (7,1) {$B_{10}$};
				\draw[very thick] (8, 0) rectangle (10, 2);
				\draw node at (9,1) {$B_{11}$};
				\draw[very thick] (11, 0) rectangle (13, 2);
				\draw node at (12,1) {$B_1$};
				\draw[very thick] (11, 2) rectangle (13, 4);
				\draw node at (12,3) {$B_2$};
				\draw[very thick] (11, 4) rectangle (13, 6);
				\draw node at (12,5) {$B_3$};
				\draw[very thick] (11, 6) rectangle (13, 8);
				\draw node at (12,7) {$B_4$};
				\draw[very thick] (11, 8) rectangle (13, 10);
				\draw node at (12,9) {$B_5$};
				\draw[very thick] (11, 10) rectangle (13, 12);
				\draw node at (12,11) {$B_6$};
			\end{tikzpicture}
			\caption[The matrix $M$ with the submatrices $N$ and $B_1$, $\dots$, $B_{11}$.]{The matrix $M$ with the submatrices $N$ and $B_1$, $\dots$, $B_{11}$. The entries of $M$ which are not in any of the submatrices are shown in grey.}\label{fig:b-part}
		\end{figure}

		Consider the 11 $(n-k) \times (n-k)$ disjoint submatrices  $B_1 , \dots, B_{11}$ of $M$ given by
		\[ B_i = \begin{cases}
				M[n - (n-k)i + 1: n - (n-k)(i-1), k + 1 : n] & 1 \leq i \leq 6,  \\
				M[k+1: n, (n-k)(i-7) + 1: (n-k)(i-6)]        & 7 \leq i \leq 11,
			\end{cases}\]
		and shown in Figure \ref{fig:b-part}. The submatrix $B_1$ contains $a_{n,n}$ and sits in the bottom right of $M$, while the others lie along the bottom and right-hand edges of $M$.

		If one of the $B_i$ satisfies $|\disc(B_i)| \leq (n-k)^2/4$, we are done by taking this submatrix as $B$, so suppose this is not the case.

		We start by using Observation \ref{obs:oneof} to show that $\disc(B_1) >  0$. Let the entries of $B_1$ be $b_{i,j}$ where $1 \leq i,j \leq n-k$. By Claim \ref{claim:tgeqnmins1}, $2t + \floor{t/2} - 2 \geq n -1$ and, applying Lemma $1$, $b_{i,i} = 1$ for all $i \leq n- k - 1$. Further, by Observation \ref{obs:oneof}, we have $b_{i,j} + b_{j,i} \geq 0$ for all $1 \leq i,j \leq n -k -1$. This means
		\[\disc(B_1) \geq (n-k - 1) - (2(n-k) - 1) = - (n-k).\]
		For $(n- k) \geq 5$, $(n-k) < (n-k)^2/4$ so we must have  $\disc(B_1) > (n-k)^2/4$.

		As $\disc(B_1) > 0 $, if $\disc(B_i) < 0$ for any $i \neq 1$, we can use an interpolation argument as in Lemma \ref{lem:submatrix} to find the claimed matrix. The argument only requires
		\[2(n-k) < \frac{(n-k)^2}{2},\]
		which is true for $(n-k) > 4$.

		We must now be in the case where $\disc(B_i) > (n-k)^2/4$ for every $i$. The bulk of the work in this case will be bounding the discrepancy of the matrix $N$, and then the discrepancy of $M$. There are $2n(n-k) - 12(n-k)^2 \leq 10(n-k)$ entries of $M$ in the gaps between the $B_i$, or in other words, there are at most $10(n-k)$ entries $a_{i,j}$ which are not contained in either $N$ or one of the $B_i$. In particular, we have
		\begin{align}
			\disc(M) & \geq \disc(N) + \disc(B_1) + \dotsb + \disc(B_{11})  - 10 (n-k) \notag \\
			         & > \disc(N) + 11 (n-k)^2/4 - 10(n-k). \label{eqn:disc}
		\end{align}

		Let $s = \min\left\{ k, t + \floor{t/2} \right\}$ so that $M[1:s, 1:s]$ is $t$ split, and let $r =  k - s$ be the number of remaining rows in $N$. Let $a_1, \dots, a_4$ be the number of 1s in $N$ guaranteed by Lemma \ref{lem:struct}, and let $a_5$ be the number of additional 1s guaranteed by also applying Observation \ref{obs:oneof}. This guarantees that at least one of $a_{i,j}$ and $a_{j,i}$ is $1$ for all $s + 1 \leq i , j \leq k$, and $a_5 \geq r(r-1)/2$.

		We have the following bounds.
		\begin{align*}
			a_1 & = s^2 - \frac{t(t+1)}{2},                     \\
			a_2 & = 2 \sum_{i=1}^r(t-i),                        \\
			a_3 & = 2 \sum_{i=1}^r \floor{\frac{t + 1 - i}{2}}, \\
			a_4 & = r,                                          \\
			a_5 & \geq \frac{r(r-1)}{2}.
		\end{align*}

		Let us first consider the case where $s = k$, so that $N$ is $t$-split.  In this case $a_2 = \dotsb = a_5 = 0$, and we can easily write down the discrepancy of $N$ as $k^2 - t(t+1)$. Since $k \geq 5n/6$, we get the bound
		\begin{align*}
			\disc(N) & \geq \frac{25n^2}{36} - t(t+1).
			\intertext{Substituting this into (\ref{eqn:disc}) and using the bounds $(n-5)/6\leq n - k \leq n/6$ we get}
			\disc(N) & > \frac{25n^2}{36} - t(t+1) + \frac{11}{4}\left( \frac{n-5}{6} \right)^2  - \frac{10n}{6} \\
			         & = \frac{1}{144} \left( 111 n^2 - 350n - 144t^2 - 144t + 275\right).
			\intertext{For $n \geq 4$, the righthand side is greater than $n^2/4$ whenever}
			t        & < \frac{1}{12} \left( \sqrt{75 n^2 - 350n + 311} - 6 \right) \approx 0.722n + o(n).
		\end{align*}
		Since we have assumed $t \leq 2n/3$, we get a contradiction for all sufficiently large $n$. In fact, we get a contradiction for all $n \geq 40$. The remaining cases need to be checked  using exact values for the floor and ceiling functions which we do with the help of a computer.

		Now we consider the case where $s = t + \floor{t/2}$ which is very similar, although more complicated. To be in this case, we must have $t + \floor{t/2} \leq k$ which implies
		\[ t + \frac{t-1}{2} \leq \frac{5(n+1)}{6},\]
		and $t \leq (5n + 8)/9 \approx 0.556n$.
		\begin{align*}
			\intertext{Start by using the bounds $(t-1)/2 \leq \floor{t/2}$ and $(t-i)/2 \leq \floor{(t + 1 - i)/2}$ to get}
			a_1 + \dotsb + a_5 & \geq \left( t + \frac{t-1}{2} \right)^2 - \frac{t(t+1)}{2} + r (2t -r - 1) + \frac{r (2t -r - 1)}{2}
			\\&\qquad  + r + \frac{r(r-1)}{2}\\
			                   & = \frac{7t^2}{4} - 2t - r^2 + 3rt - r  + \frac{1}{4}.
			\intertext{By definition, $r = k - t - \floor{t/2}$, so we get the bounds $5n/6 - t - t/2 \leq r \leq 5(n+1)/6 - t - (t-1)/2$, and substituting these in gives}
			a_1 + \dotsb + a_5 & \geq \frac{7}{4} t^2 - 2t + \frac{1}{4} -  \left( \frac{5(n+1)}{6} - t - \frac{t-1}{2}\right)^2 + 3t \left( \frac{5n}{6} -t - \frac{t}{2} \right) \\&\qquad - \left( \frac{5(n+1)}{6} - t - \frac{t-1}{2} \right)\\
			                   & = \frac{1}{36} \left( - 25n^2 + 180nt - 110n - 180 t^2 + 126t - 103 \right)
			\intertext{Plugging this into (\ref{eqn:disc}) and using the bounds $5n/6 \leq k \leq 5(n+1)/6$ we get}
			\disc(M)           & >2 (a_1 + \dotsb a_5) - \left( \frac{5(n+1)}{6} \right)^2 + \frac{11}{4} \left( \frac{n-5}{6} \right)^2 - \frac{10n}{6}                           \\
			                   & \geq \frac{1}{144} \left( - 289n^2 + 1440nt - 1430n - 1440t^2 + 1008t - 649 \right).
		\end{align*}
		When $n \geq 27$, this is greater than $n^2/4$ whenever
		\begin{align*}
			 & \frac{1}{120}\left(60n + 42 - \sqrt{350 n^2 - 9260n - 4726}\right) <           \\
			 & \qquad t < \frac{1}{120}\left(60n + 42 + \sqrt{350 n^2 - 9260n - 4726}\right),
		\end{align*}
		or approximately,
		\[0.344n < t < 0.656n.\]
		We have the bounds
		\[ \frac{1}{4} \left( \sqrt{3n^2 - 6n + 7} - 2 \right) \leq  t \leq \frac{5n + 8}{9}, \]
		and so, for $n \geq 44$, $\disc(M) > n^2/4$.

		This again leaves a few cases which we check with the help of a computer.
	\end{proof}

	Given a submatrix $B$ as in the above claim we apply the induction hypothesis, noting that $(n-k) \geq 5$ since $n \geq 30$, to find that $B$ is split. Let $C$ be the split submatrix obtained from applying Lemma 4 to $B$, and let $C$ be $\ell$-split up to rotation. Note that $\ell \geq 3$ as $(n-k) \geq 5$ and $|\disc(B)| \leq (n-k)^2/4$, and we can assume $\ell \leq 2n/3$ as $M$ is not split.

	Hence, $C$ contains exactly one of $a_{1,1}$, $a_{1,n}$, $a_{n,1}$ and $a_{n,n}$, and we will split into cases based on which one $C$ contains. We will also sometimes need to consider cases for whether the entry is $1$ or $-1$, but in all cases we will find a contradiction.

	From Lemma \ref{lem:struct} applied to $M'$ and Claim \ref{claim:tgeqnmins1}, we already know some of the entries and we highlight some important entries in the following claim.

	\begin{claim}\label{claim:particular1s}
		We have
		\begin{enumerate}
			\item $a_{j,1} = a_{1, j} = \begin{cases}
					      1  & t + 1 \leq j \leq n-1, \\
					      -1 & 1 \leq j \leq t,
				      \end{cases}$
			\item $a_{2,t} = a_{t,2} = 1$,
			\item $a_{i,i} = 1$ for all $(t+2)/2 \leq i \leq n -1$.
		\end{enumerate}
	\end{claim}

	Suppose the submatrix $C$ contains $a_{1,1}$ so sits in the top-left corner. Since $M[1:t + \floor{t/2}, 1:t + \floor{t/2}]$ is $t$-split, $C$ must also be $t$-split. As $C$ was found by applying Lemma \ref{lem:struct} to $B$, it must contain a $-1$ from $B$. Hence, $t \geq 5n/6$ which is a contradiction as we assumed that $t \leq  2n/3$.

	Suppose instead that $C$ contains $a_{1,n}$ so sits in the top-right corner. Since $\ell \geq 3$, if the corner entry is $-1$, so is the entry $a_{1,n-1}$, but this contradicts Claim \ref{claim:particular1s}. Suppose instead that the corner entry is $1$. Since $C$ is $\ell$-split up to rotation we have, for all $1 \leq i, (n - j  + 1) \leq \ell + \floor{\ell/2}$,
	\begin{equation}\label{eqn:C}
		a_{i,j} = \begin{cases}
			-1 & i + (n -j + 1) \geq \ell + 2, \\
			1  & \text{otherwise}.
		\end{cases}
	\end{equation}

	If $n - \ell > t$, then $a_{1, n-\ell} = -1$ by (\ref{eqn:C}) and $a_{1, n-\ell} = 1$ by Claim \ref{claim:particular1s}.
	Suppose $n - \ell < t$. Then $a_{1, t} = 1$ by (\ref{eqn:C}) and $a_{1, t} = -1$ as  $M[1:t, 1:t]$ is $t$-split.
	Finally, when $n-\ell = t$, we have $a_{2,t} = -1$ by (\ref{eqn:C}) and $a_{2,t} = 1$ from Claim \ref{claim:particular1s}. Some illustrative examples of these three cases are shown in Figure \ref{fig:case-1-n}.

	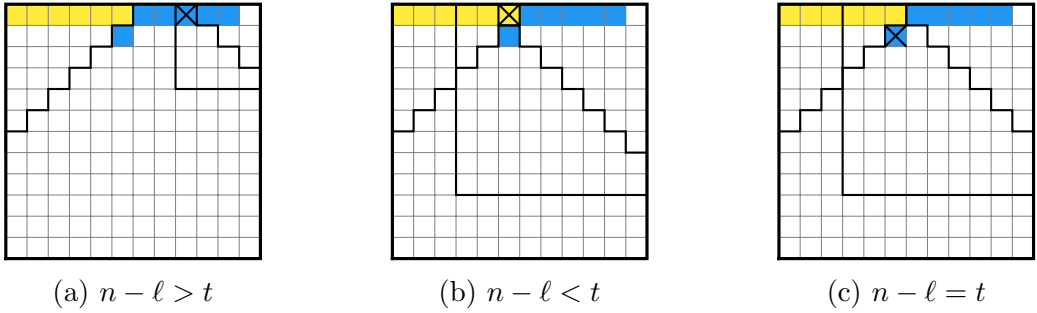
\begin{figure}
	\centering
	\begin{subfigure}{0.25\textwidth}
		\centering
		\begin{tikzpicture}[scale=(\textwidth-1.2pt)/12cm]

			\fill[fill=yellow5] (0, 11) rectangle(6, 12);
			\fill[fill=blue5] (6, 11) rectangle(11, 12);
			\fill[fill=blue5] (5, 10) rectangle(6,11);
			\draw[very thin, gray] (0,0) grid (12,12);
			\draw[thick] (6,12)
			\foreach \myvar in {6, 5, 4, ..., 0}{%
					-- (\myvar, 6 + \myvar) -- (\myvar, 5 + \myvar)
				} -- (0, 12) -- (6,12);

			\draw[thick] (9,12)
			\foreach \myvar in {9, 10, ..., 12}{%
					-- (\myvar, 21 -  \myvar) -- (\myvar, 20 - \myvar)
				};

			\draw[thick] (8,8) rectangle (12, 12);

			\draw[thick] (8, 12) -- (9,11);
			\draw[thick] (9, 12) -- (8,11);

			\draw[very thick, cap=rect] (0,0) -- (12,0) -- (12,12) -- (0,12) --(0,0);

		\end{tikzpicture}
		\caption{$n- \ell > t$}
	\end{subfigure}\hfill%
	\begin{subfigure}{0.25\textwidth}
		\centering
		\begin{tikzpicture}[scale=(\textwidth-1.2pt)/12cm]

			\fill[fill=yellow5] (0, 11) rectangle(6, 12);
			\fill[fill=blue5] (6, 11) rectangle(11, 12);
			\fill[fill=blue5] (5, 10) rectangle(6,11);
			\draw[very thin, gray] (0,0) grid (12,12);
			\draw[thick] (6,12)
			\foreach \myvar in {6, 5, 4, ..., 0}{%
					-- (\myvar, 6 + \myvar) -- (\myvar, 5 + \myvar)
				} -- (0, 12) -- (6,12);

			\draw[thick] (9,12)
			\foreach \myvar in {5, 6, ..., 12}{%
					-- (\myvar, 17 -  \myvar) -- (\myvar, 16 - \myvar)
				};

			\draw[thick] (3,3) rectangle (12, 12);

			\draw[thick] (5, 12) -- (6,11);
			\draw[thick] (6, 12) -- (5,11);

			\draw[very thick, cap=rect] (0,0) -- (12,0) -- (12,12) -- (0,12) --(0,0);

		\end{tikzpicture}
		\caption{$n- \ell < t$}
	\end{subfigure}\hfill%
	\centering
	\begin{subfigure}{0.25\textwidth}
		\centering
		\begin{tikzpicture}[scale=(\textwidth-1.2pt)/12cm]

			\fill[fill=yellow5] (0, 11) rectangle(6, 12);
			\fill[fill=blue5] (6, 11) rectangle(11, 12);
			\fill[fill=blue5] (5, 10) rectangle(6,11);
			\draw[very thin, gray] (0,0) grid (12,12);
			\draw[thick] (6,12)
			\foreach \myvar in {6, 5, 4, ..., 0}{%
					-- (\myvar, 6 + \myvar) -- (\myvar, 5 + \myvar)
				} -- (0, 12) -- (6,12);

			\draw[thick] (9,12)
			\foreach \myvar in {6, 7, ..., 12}{%
					-- (\myvar, 18 -  \myvar) -- (\myvar, 17 - \myvar)
				};

			\draw[thick] (3,3) rectangle (12, 12);

			\draw[thick] (5, 11) -- (6,10);
			\draw[thick] (6, 11) -- (5,10);

			\draw[very thick, cap=rect] (0,0) -- (12,0) -- (12,12) -- (0,12) --(0,0);

		\end{tikzpicture}
		\caption{$n - \ell =t$}
	\end{subfigure}%
	\caption[The three possible cases when $C$ contains $a_{1,n}$ and $a_{1,n} = 1$.]{The three cases when $C$ contains $a_{1,n}$ and $a_{1,n} = 1$. The yellow squares represent some of the $a_{i,j}$ which are known to be $-1$ from Claim \ref{claim:particular1s} and the blue squares those which are $1$. The square which gives the contradiction is marked with a cross.}
	\label{fig:case-1-n}
\end{figure}

	The case where $C$ contains $a_{n,1}$ is done in the same way with the rows and columns swapped.

	This leaves the case where $C$ contains $a_{n,n}$. Since $\ell \geq 3$, if the entry $a_{n, n}$ equals $-1$, so does the entry $a_{n-1, n-1}$, and this contradicts Claim \ref{claim:particular1s}. If instead $a_{n,n} = 1$, we consider the entry $a_{i, i}$ where $i = n  + 1 - \ceil{(\ell+2)/2}$, which must be $-1$. However, since $\ell \leq 2n/3$, \[n  + 1 - \ceil{(\ell+2)/2}\geq \frac{2n}{3} - \frac{1}{2} > \frac{n}{3} + 1 \geq \frac{t+2}{2},\] and $a_{i,i} = 1$ by Claim \ref{claim:particular1s}. This final contradiction is shown in Figure \ref{fig:case-n-n}.
	\begin{figure}
	\centering
	\begin{tikzpicture}[scale=0.4\textwidth/12cm]

		\foreach \myvar in {4,...,11}{%
		\fill[blue5] (\myvar-1, 12-\myvar) rectangle +(1,1);}
		
		\draw[very thin, gray] (0,0) grid (12,12);
		\draw[thick] (6,12)
		\foreach \myvar in {6, 5, 4, ..., 0}{%
			-- (\myvar, 6 + \myvar) -- (\myvar, 5 + \myvar)
		} -- (0, 12) -- (6,12);
		
		\draw[thick] (7,0)
		\foreach \myvar in {7,8, ..., 12}{%
			-- (\myvar, \myvar - 7) -- (\myvar, \myvar - 6)
		};
		
		\draw[thick] (5,0) rectangle (12, 7);

		\draw[thick] (8, 3) -- (9,4);
		\draw[thick] (8, 4) -- (9,3);
		
		\draw[very thick, cap=rect] (0,0) -- (12,0) -- (12,12) -- (0,12) --(0,0);
		
	\end{tikzpicture}
\caption[The case where $C$ contains $a_{n,n}$ and $a_{n,n} = 1$.]{The case where $C$ contains $a_{n,n}$ and $a_{n,n} = 1$. The square marked with a cross gives a contradiction.}
\label{fig:case-n-n}
\end{figure}
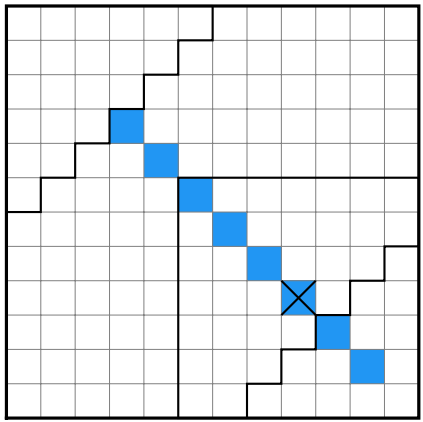
\end{proof}

We remark that it should be possible to improve the bound $n^2/4$ using a similar proof provided one can check a large enough base case. Indeed, we believe that all the steps in the above proof hold when the bound is increased to $n^2/3$, but only when $n$ is large enough. For example, Claim \ref{claim:tgeqnmins1} fails for $n = 127$ and our proof of Claim \ref{claim:B} fails for $n = 67$. Checking base cases this large is far beyond the reach of our computer check, and some new ideas would be needed here.

\section{Open problems}\label{sec:open-problems}

The main open problem is to determine the correct lower bound for the (absolute value of the) discrepancy of a non-split $\{-1,1\}$-matrix with no zero-sum squares. We have improved the lower bound to $\floor{n^2/4} + 1$, but this does not appear to be optimal.

\begin{figure}
	\centering
	\begin{tikzpicture}[scale=0.5\textwidth/9cm]

		\fill[yellow5] (0,0) rectangle (1,1);
		\fill[blue5] (1,0) rectangle (2,1);
		\fill[yellow5] (2,0) rectangle (3,1);
		\fill[blue5] (3,0) rectangle (4,1);
		\fill[yellow5] (4,0) rectangle (5,1);
		\fill[blue5] (5,0) rectangle (6,1);
		\fill[yellow5] (6,0) rectangle (7,1);
		\fill[blue5] (7,0) rectangle (8,1);
		\fill[yellow5] (8,0) rectangle (9,1);

		\fill[blue5] (0,1) rectangle (1,2);
		\fill[blue5] (1,1) rectangle (2,2);
		\fill[blue5] (2,1) rectangle (3,2);
		\fill[blue5] (3,1) rectangle (4,2);
		\fill[blue5] (4,1) rectangle (5,2);
		\fill[blue5] (5,1) rectangle (6,2);
		\fill[blue5] (6,1) rectangle (7,2);
		\fill[blue5] (7,1) rectangle (8,2);
		\fill[blue5] (8,1) rectangle (9,2);

		\fill[yellow5] (0,2) rectangle (1,3);
		\fill[blue5] (1,2) rectangle (2,3);
		\fill[yellow5] (2,2) rectangle (3,3);
		\fill[blue5] (3,2) rectangle (4,3);
		\fill[yellow5] (4,2) rectangle (5,3);
		\fill[blue5] (5,2) rectangle (6,3);
		\fill[yellow5] (6,2) rectangle (7,3);
		\fill[blue5] (7,2) rectangle (8,3);
		\fill[yellow5] (8,2) rectangle (9,3);

		\fill[blue5] (0,3) rectangle (1,4);
		\fill[blue5] (1,3) rectangle (2,4);
		\fill[blue5] (2,3) rectangle (3,4);
		\fill[blue5] (3,3) rectangle (4,4);
		\fill[blue5] (4,3) rectangle (5,4);
		\fill[blue5] (5,3) rectangle (6,4);
		\fill[blue5] (6,3) rectangle (7,4);
		\fill[blue5] (7,3) rectangle (8,4);
		\fill[blue5] (8,3) rectangle (9,4);

		\fill[yellow5] (0,4) rectangle (1,5);
		\fill[blue5] (1,4) rectangle (2,5);
		\fill[yellow5] (2,4) rectangle (3,5);
		\fill[blue5] (3,4) rectangle (4,5);
		\fill[yellow5] (4,4) rectangle (5,5);
		\fill[blue5] (5,4) rectangle (6,5);
		\fill[yellow5] (6,4) rectangle (7,5);
		\fill[blue5] (7,4) rectangle (8,5);
		\fill[yellow5] (8,4) rectangle (9,5);

		\fill[blue5] (0,5) rectangle (1,6);
		\fill[blue5] (1,5) rectangle (2,6);
		\fill[blue5] (2,5) rectangle (3,6);
		\fill[blue5] (3,5) rectangle (4,6);
		\fill[blue5] (4,5) rectangle (5,6);
		\fill[blue5] (5,5) rectangle (6,6);
		\fill[blue5] (6,5) rectangle (7,6);
		\fill[blue5] (7,5) rectangle (8,6);
		\fill[blue5] (8,5) rectangle (9,6);

		\fill[yellow5] (0,6) rectangle (1,7);
		\fill[blue5] (1,6) rectangle (2,7);
		\fill[yellow5] (2,6) rectangle (3,7);
		\fill[blue5] (3,6) rectangle (4,7);
		\fill[yellow5] (4,6) rectangle (5,7);
		\fill[blue5] (5,6) rectangle (6,7);
		\fill[yellow5] (6,6) rectangle (7,7);
		\fill[blue5] (7,6) rectangle (8,7);
		\fill[yellow5] (8,6) rectangle (9,7);

		\fill[blue5] (0,7) rectangle (1,8);
		\fill[blue5] (1,7) rectangle (2,8);
		\fill[blue5] (2,7) rectangle (3,8);
		\fill[blue5] (3,7) rectangle (4,8);
		\fill[blue5] (4,7) rectangle (5,8);
		\fill[blue5] (5,7) rectangle (6,8);
		\fill[blue5] (6,7) rectangle (7,8);
		\fill[blue5] (7,7) rectangle (8,8);
		\fill[blue5] (8,7) rectangle (9,8);

		\fill[yellow5] (0,8) rectangle (1,9);
		\fill[blue5] (1,8) rectangle (2,9);
		\fill[yellow5] (2,8) rectangle (3,9);
		\fill[blue5] (3,8) rectangle (4,9);
		\fill[yellow5] (4,8) rectangle (5,9);
		\fill[blue5] (5,8) rectangle (6,9);
		\fill[yellow5] (6,8) rectangle (7,9);
		\fill[blue5] (7,8) rectangle (8,9);
		\fill[yellow5] (8,8) rectangle (9,9);

		\draw[step=1, thin] (0,0) grid (9,9);

		\draw[very thick] (0,0) rectangle(9,9);
	\end{tikzpicture}

	\caption[A $9 \times 9$ $\{-1,1\}$-matrix with no zero-sum squares and discrepancy 31. Amongst all $\{-1,1\}$ matrices which are neither split nor contain a zero-sum square this has the smallest (in magnitude) discrepancy.]{A $9 \times 9$ $\{-1,1\}$-matrix with no zero-sum squares and discrepancy 31. Amongst all $\{-1,1\}$ matrices which are neither split nor contain a zero-sum square this has the smallest (in magnitude) discrepancy. The yellow squares represent a $-1$ and the blue squares represent a $1$. }
	\label{fig:conjecture}
\end{figure}
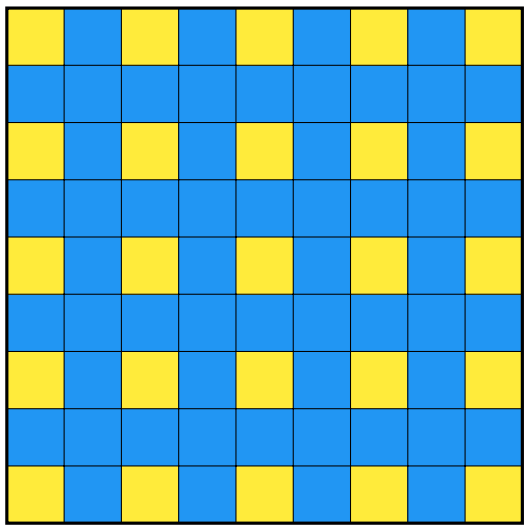

The best known construction is the following example by Ar\'evalo, Montejano and Rold\'an-Pensado \cite{arevalo2021zero}. Let $M = \left(a_{i,j}\right)$ be given by \[a_{i,j} = \begin{cases}
		-1 & \text{$i$ and $j$ are odd}, \\
		1  & \text{otherwise}.
	\end{cases}\]
This has discrepancy $n^2/2$ when $n$ is even and $(n-1)^2/2 - 1$ when $n$ is odd. With the help of a computer we have verified that this construction is best possible when $9 \leq n \leq 32$, and we conjecture that this holds true for all $n \geq 9$. In fact, our computer search shows that the above example is the unique non-split matrix avoiding zero-sum squares with minimum (in magnitude) discrepancy, up to reflections and multiplying by $-1$. An example when $n = 9$ is given in Figure \ref{fig:conjecture}.

We note that the condition $n \geq 9$ is necessary, as shown by the $8 \times 8$ matrix with discrepancy 30 given in Figure \ref{fig:counter}.
\begin{conjecture}
	Let $n \geq 9$. Every $n \times n$ non-split $\{-1, 1\}$-matrix $M$ with \[|\disc(M)| \leq \begin{cases}
			\frac{n^2}{2} - 1    & \text{$n$ is even} \\
			\frac{(n-1)^2}{2} -2 & \text{$n$ is odd}
		\end{cases}\]
	contains a zero-sum square.
\end{conjecture}

\begin{figure}
	\centering
	\begin{tikzpicture}[scale=0.5\textwidth/8cm]

		\fill[blue5] (0,0) rectangle (1,1);
		\fill[blue5] (1,0) rectangle (2,1);
		\fill[yellow5] (2,0) rectangle (3,1);
		\fill[blue5] (3,0) rectangle (4,1);
		\fill[blue5] (4,0) rectangle (5,1);
		\fill[blue5] (5,0) rectangle (6,1);
		\fill[blue5] (6,0) rectangle (7,1);
		\fill[blue5] (7,0) rectangle (8,1);

		\fill[blue5] (0,1) rectangle (1,2);
		\fill[blue5] (1,1) rectangle (2,2);
		\fill[blue5] (2,1) rectangle (3,2);
		\fill[blue5] (3,1) rectangle (4,2);
		\fill[blue5] (4,1) rectangle (5,2);
		\fill[blue5] (5,1) rectangle (6,2);
		\fill[blue5] (6,1) rectangle (7,2);
		\fill[blue5] (7,1) rectangle (8,2);

		\fill[blue5] (0,2) rectangle (1,3);
		\fill[blue5] (1,2) rectangle (2,3);
		\fill[blue5] (2,2) rectangle (3,3);
		\fill[blue5] (3,2) rectangle (4,3);
		\fill[blue5] (4,2) rectangle (5,3);
		\fill[blue5] (5,2) rectangle (6,3);
		\fill[blue5] (6,2) rectangle (7,3);
		\fill[blue5] (7,2) rectangle (8,3);

		\fill[yellow5] (0,3) rectangle (1,4);
		\fill[blue5] (1,3) rectangle (2,4);
		\fill[blue5] (2,3) rectangle (3,4);
		\fill[blue5] (3,3) rectangle (4,4);
		\fill[blue5] (4,3) rectangle (5,4);
		\fill[blue5] (5,3) rectangle (6,4);
		\fill[blue5] (6,3) rectangle (7,4);
		\fill[blue5] (7,3) rectangle (8,4);

		\fill[yellow5] (0,4) rectangle (1,5);
		\fill[yellow5] (1,4) rectangle (2,5);
		\fill[blue5] (2,4) rectangle (3,5);
		\fill[blue5] (3,4) rectangle (4,5);
		\fill[blue5] (4,4) rectangle (5,5);
		\fill[blue5] (5,4) rectangle (6,5);
		\fill[blue5] (6,4) rectangle (7,5);
		\fill[blue5] (7,4) rectangle (8,5);

		\fill[yellow5] (0,5) rectangle (1,6);
		\fill[yellow5] (1,5) rectangle (2,6);
		\fill[yellow5] (2,5) rectangle (3,6);
		\fill[blue5] (3,5) rectangle (4,6);
		\fill[blue5] (4,5) rectangle (5,6);
		\fill[blue5] (5,5) rectangle (6,6);
		\fill[blue5] (6,5) rectangle (7,6);
		\fill[yellow5] (7,5) rectangle (8,6);

		\fill[yellow5] (0,6) rectangle (1,7);
		\fill[yellow5] (1,6) rectangle (2,7);
		\fill[yellow5] (2,6) rectangle (3,7);
		\fill[yellow5] (3,6) rectangle (4,7);
		\fill[blue5] (4,6) rectangle (5,7);
		\fill[blue5] (5,6) rectangle (6,7);
		\fill[blue5] (6,6) rectangle (7,7);
		\fill[blue5] (7,6) rectangle (8,7);

		\fill[yellow5] (0,7) rectangle (1,8);
		\fill[yellow5] (1,7) rectangle (2,8);
		\fill[yellow5] (2,7) rectangle (3,8);
		\fill[yellow5] (3,7) rectangle (4,8);
		\fill[yellow5] (4,7) rectangle (5,8);
		\fill[blue5] (5,7) rectangle (6,8);
		\fill[blue5] (6,7) rectangle (7,8);
		\fill[blue5] (7,7) rectangle (8,8);

		\draw[step=1, thin] (0,0) grid (8,8);

		\draw[very thick] (0,0) rectangle(8,8);
	\end{tikzpicture}

	\caption[An $8 \times 8$ $\{-1,1\}$-matrix with no zero-sum squares and discrepancy 30.]{An $8 \times 8$ $\{-1,1\}$-matrix with no zero-sum squares and discrepancy 30. The yellow squares represent a $-1$ and the blue squares represent a $1$.}
	\label{fig:counter}
\end{figure}
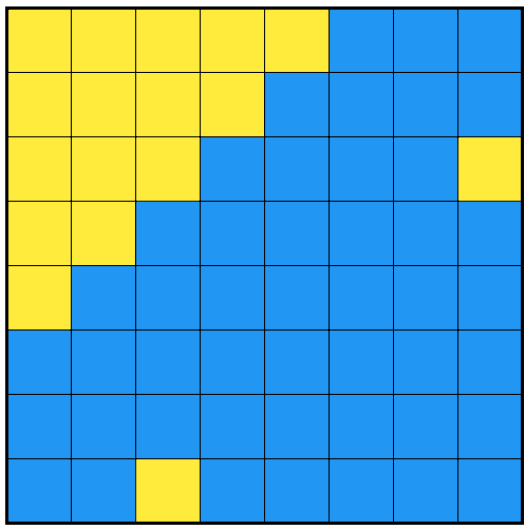

Ar\'evalo, Montejano and Rold\'an-Pensado prove their result for both $n \times n$ and $n \times (n+1)$ matrices, and computational experiments suggest that Theorem \ref{thm:low-bound} holds for $n \times (n+1)$ matrices as well. More generally, what is the best lower bound for a general $n \times m$ matrix when $n$ and $m$ are large?

\begin{problem}
Let $f(n, m)$ be the minimum $d \in \mathbb{N}$ such that there exists an $n \times m$ non-split $\{-1,1\}$ matrix $M$ with $|\disc(M)| \leq d$. What are the asymptotics of $f(n,m)$?
\end{problem}

\paragraph{Acknowledgements} The author would like to thank the anonymous referee for their helpful comments.

\bibliography{zero-sum}
\bibliographystyle{abbrev-bold}

\end{document}